\documentclass[12pt]{article}
\usepackage[english]{babel}
\usepackage{CJKutf8}
\usepackage{authblk}

\usepackage{amsthm}
\usepackage{amssymb}
\usepackage{amsmath}
\usepackage{mathtools}
\usepackage{mathrsfs}
\usepackage{stix}
\usepackage{enumerate}
\usepackage{nicematrix}
\NiceMatrixOptions{cell-space-top-limit=4pt,cell-space-bottom-limit=4pt}
\usepackage{xcolor}
\usepackage{caption}
\usepackage{subcaption}  % enables subtable

\usepackage{algorithm}
\usepackage{algpseudocode} % or algcompatible, algruled, etc.

\usepackage[top=2cm,bottom=2cm,left=2cm,right=2cm,a4paper]{geometry}
\usepackage{setspace}

\usepackage{enumitem}
\setlist[itemize]{noitemsep, topsep=0pt}

\newtheorem{theorem}{\textbf{Theorem}}[section]
\newtheorem{definition}[theorem]{\textbf{Definition}}
\newtheorem{proposition}[theorem]{\textbf{Proposition}}
\newtheorem{lemma}[theorem]{\bf Lemma}
\newtheorem{corollary}[theorem]{\bf Corollary}

\theoremstyle{definition}

\newtheorem*{question}{\bf Question}

\usepackage{comment}

\newcommand{\will}[1]{\textcolor{purple}{#1}}
\usepackage{ulem}

\usepackage{tikz-cd} 

\usepackage{listings}

\usepackage{amsrefs}
\usepackage[toc,page]{appendix}

\DeclareMathOperator{\lcm}{lcm}
\DeclareMathOperator{\ord}{ord}

\allowdisplaybreaks
\title{Periodicity and Dynamical Systems of Dickson Polynomials in Finite Fields}
\author[1]{Yen-Ju, Chen}
\author[2]{Wayne, Peng}
\affil[1]{Department of Mathematics, National Taiwan Normal University}
\affil[2]{Department of Mathematics, National Central University}
\date{\today}

\begin{document}

\maketitle
\begin{abstract}
This paper investigates the dynamical properties of the Dickson polynomials \(D_n(x, \alpha)\) of the first kind over finite fields, with an emphasis on the periodicity and algebraic structure of their iterated sequences. We consider the sequence \( [D_n(x, \alpha) \pmod{x^q - x}]_{n\geq1} \), and determine the exact period of this sequence. We then use the classical functional equation for the Dickson polynomials to relate the dynamics of \(D_n(x,\alpha)\) to the power map \(u\mapsto u^n\) on a suitable subset of \(\mathbb F_{q^2}^\times\). In the permutation case \(\gcd(n,q^2-1)=1\), this gives an explicit description of the group of Dickson polynomial functions under composition. We also obtain partial structural result in the non-permutation case. As further applications, we derive several new identities for the Dickson polynomials. Finally, we identify and prove a symmetry property of the Dickson polynomial family.
\end{abstract}

%%%%%%%%%%%%%%%%%%%%%%%%%%%%%%%%%%%%%%%%%%
\section{Introduction}
%%%%%%%%%%%%%%%%%%%%%%%%%%%%%%%%%%%%%%%%%%

Let \(\mathbb F_q\) be the finite field of order \(q\). The iteration of polynomial and rational maps over $\mathbb F_q$ has been studied extensively from several points of view, including finite dynamical systems, pseudorandom number generation, cryptography, and the analysis of Pollard-rho type algorithms. A standard way to study a map
\[
    f:\mathbb F_q \longrightarrow \mathbb F_q
\]
is through its functional graph: the vertices are the elements of \(\mathbb F_q\), and there is a directed edge from \(a\) to \(f(a)\). Since the functional graph is defined by a function on a finite set, each connected component consists of a unique directed cycle together with rooted trees attached to the periodic points, points on cycles. Thus quantities such as periods, preperiods, rho lengths, and connected components can be studied through the associated functional graph.

A broad survey of iterations of mappings over finite fields is given by Martins, Panario, and Qureshi~\cite{MartinsPanarioQureshiSurvey}. Their survey discusses, among other families, power maps, R\'edei functions, Chebyshev polynomials, linearized polynomials, and quadratic maps. Some of these families are especially tractable because their dynamics can be compared with the power map \(x\mapsto x^n\) on \(\mathbb F_q^\times\). Similar group-theoretic mechanisms appear in the study of R\'edei functions; see, for instance,~\cites{QureshiPanarioRedei,QureshiPanarioMartinsRedei}. The Chebyshev case is also closely related to this framework; see~\cites{ALDENGASSERT201483,QureshiPanarioChebyshev}. The group-theoretic viewpoint is one of the main ideas of the present paper.

Explicit descriptions of functional graphs have been obtained for several trackable families. For example, Vasiga and Shallit studied the functional graphs of \(x^2\) and \(x^2-2\) over prime fields~\cite{VasigaShallit}. Qureshi and Panario completely described the functional graphs of Chebyshev polynomials over finite fields, including explicit information about periodic points, preperiodic points, rooted trees attached to periodic points and connected components~\cite{QureshiPanarioChebyshev}. On the other hand, for less symmetric families, random mapping heuristics often provide a useful model for understanding average dynamical behavior; see, for example,~\cites{PollardRhoFactorization,BrentPollardFermat,BachPollardRho,MartinsPanarioRandomMappings}. Periodic-point statistics for reductions modulo primes have also been studied using arboreal Galois theory~\cites{Odoni1985,BIJRMMST,JKT2016,Juul2021}.

Among these families, Chebyshev polynomials are particularly relevant to the present work. The Chebyshev polynomial \(T_n\) is characterized by the identity
\[
    T_n(u+u^{-1}) = u^n+u^{-n}.
\]
Dickson polynomials are classical objects closely related to Chebyshev polynomials. For \(\alpha\in \mathbb F_q\), the Dickson polynomial \(D_n(x,\alpha)\) is characterized by the functional equation
\begin{equation}\label{functional equation}
    D_n\left(\phi_{\alpha}(u),\alpha\right)
    =\phi_{\alpha^n}(u^n)
\end{equation}
where $\phi_{\beta}(u)=u+\frac{\beta}{u}$. The standard algebraic properties of Dickson polynomials, including their recurrence relations, explicit formulas, composition laws, and permutation criteria, are well known; see~\cites{LidlMullenTurnwald1993,Lidl,ChouGomezMullen1988,WANG2012814}. General background on finite fields and permutation polynomials can be found in~\cites{lidl1997finite,mullen2013handbook,Hou2015,YuanDing2014}.

The purpose of this paper is different from the existing literature that studies the functional graph of a fixed polynomial map on \(\mathbb F_q\). Instead, we study the sequential behavior of Dickson polynomial functions as the index \(n\) varies. More precisely, we determine the exact period of the sequence
\[
    \left[D_n(x,\alpha)\pmod{x^q-x}\right]_{n\geq 1}
\]
in the ring of polynomial functions \(\mathbb F_q[x]/(x^q-x)\).

Our first main result is the following exact periodicity theorem.

\begin{theorem}
\label{thm: exact period of Dickson polynomials}
Let \(q\) be a prime power and let \(\alpha\in \mathbb F_q^\times\). The period of the sequence
\[
    \left[D_n(x,\alpha)\pmod{x^q-x}\right]_{n\geq 1}
\]
is \(\frac{q^2-1}{2}\) if \(q\) is odd and \(\alpha\) is a square in \(\mathbb F_q\), and is \(q^2-1\) otherwise.
\end{theorem}

We use this theorem and Equation~\eqref{functional equation} to study of composition dynamics when \(\alpha^n=\alpha\). Under this assumption, Equation~\eqref{functional equation} can be represented by the commutative diagram
\begin{equation}\label{commutative diagram}
\begin{tikzcd}
\mathcal T_{\alpha} \arrow{r}{u\mapsto u^n} \arrow[swap]{d}{\phi_{\alpha}}
&
\mathcal T_{\alpha} \arrow{d}{\phi_{\alpha}}
\\
\mathbb F_q \arrow{r}{D_n(x,\alpha)}
&
\mathbb F_q
\end{tikzcd}
\end{equation}
where $\mathcal T_\alpha=\phi_\alpha^{-1}(\mathbb F_q)$. This diagram expresses the Dickson polynomial dynamics as a quotient of the power map \(u\mapsto u^n\). If \(m\) and \(n\) both satisfy \(\alpha^m=\alpha^n=\alpha\), then the Dickson composition law gives
\[
    D_m(D_n(x,\alpha),\alpha)
    =
    D_{mn}(x,\alpha)
    =
    D_n(D_m(x,\alpha),\alpha).
\]
Thus composition of Dickson polynomials corresponds to the multiplication of their indices.

For notational convenience, we write
\[
    D_{n,\alpha}(x)=D_n(x,\alpha).
\]
Let \(\pi_q(\alpha)\) denote the period of the sequence
\[
    \left[D_n(x,\alpha)\pmod{x^q-x}\right]_{n\geq 1}
\]
given in Theorem~\ref{thm: exact period of Dickson polynomials}. Then the map
\begin{equation}\label{the map}
    f:n \longmapsto D_{n,\alpha}(x)\pmod{x^q-x}
\end{equation}
is compatible with multiplication of indices modulo \(\pi_q(\alpha)\) and composition of polynomial functions. In the permutation case, namely when $\gcd(n,q^2-1)=1$, $f$ is a group homomorphism from a subgroup of $\left(\mathbb Z/\pi_q(\alpha)\mathbb Z\right)^\times$ to the group of permutations of \(\mathbb F_q\) generated by Dickson polynomial functions. We determine this homomorphism explicitly.

\begin{theorem}\label{thm: group structure}
Let \(\alpha\in\mathbb F_q^\times\), and let
\[
    \mathcal D_{\alpha}
    =
    \left\{
        D_{n,\alpha}(x)\pmod{x^q-x}
        \mid
        \alpha^n=\alpha\text{ and }\gcd(n,q^2-1)=1
    \right\}.
\]
Then \(\mathcal D_{\alpha}\) is a group under composition. Moreover, let $U_\alpha=\{n\in(\mathbb{Z}/\pi_q(\alpha)\mathbb Z)^\times|n\equiv 1\pmod r\}$ where $r$ is the order of $\alpha$ in $\mathbb F_q^\times\setminus\{1\}$. Then, \(\mathcal D_{\alpha}\) is isomorphic to
\[
\begin{cases}
    (\mathbb Z/\pi_q(1)\mathbb Z)^\times/H_1,
        &\alpha=1;\\[4pt]
    U_\alpha/H_2,
        &\text{otherwise,}
\end{cases}
\]
where \(H_1\cong \mathbb Z/2\mathbb Z\oplus \mathbb Z/2\mathbb Z\) and \(H_2\cong \mathbb Z/2\mathbb Z\).
\end{theorem}

Although our main results are stated in terms of the eventual periodicity of sequences of polynomial functions, they also have a direct interpretation in terms of functional graphs. Fix \(n\), and suppose that \(\mathscr l\) and \(\mathscr k\) are chosen minimally so that
\[
D_{n,\alpha}^{\mathscr l+\mathscr k}(x)\equiv D_{n,\alpha}^{\mathscr l}(x)\pmod{x^q-x},
\]
where the superscript denotes iteration. Then \(\mathscr l\) is the maximal preperiodic length among all points of \(\mathbb F_q\), and \(\mathscr k\) is the least common multiple of the cycle lengths in the functional graph of \(D_{n,\alpha}\). Thus \(\mathscr l+\mathscr k\) is a natural global rho-length parameter for \(D_{n,\alpha}\), while the pointwise rho length \(\rho(x)\) may be smaller for individual points.

During the investigation of this topic, we also obtain a new symmetry identity for the coefficients of Dickson polynomials.

\begin{theorem}\label{main theorem: 1}
Let \(q\) be a power of an odd prime and let \(\alpha \in\mathbb F_q^\times\). Then
\[
    x^{q^2+1}\left(D_{q^2-1}\left(x^{-1},\alpha\right)-2\right)
    =
    (-4\alpha)^{-1}
    \left(D_{q^2-1}\left(x,(16\alpha)^{-1}\right)-2\right).
\]
Furthermore, if \(\alpha\) is a square in \(\mathbb F_q^\times\), then
\[
    x^{\frac{q^2+1}{2}}
    D_{\frac{q^2-1}{2}}\left(x^{-1},\alpha\right)
    =
    2D_{\frac{q^2+1}{2}}\left(x,(16\alpha)^{-1}\right).
\]
\end{theorem}

This identity reflects a nontrivial symmetry in the coefficients of Dickson polynomials and motivates the coefficient-comparison viewpoint used in the algorithmic appendix.

The paper is organized as follows. Section~2 recalls the necessary preliminaries from finite fields, combinatorics, and elementary number theory used in the proofs. Sections~3 and~4 establish Theorems~\ref{thm: exact period of Dickson polynomials} and~\ref{thm: group structure}. Section~5 proves the coefficient symmetry identity in Theorem~\ref{main theorem: 1}, and leave the algorithms in Appendix.

%%%%%%%%%%%%%%%%%%%%%%%%%%%%%%%%%%%%%%%%%%
\section{Preliminary}
%%%%%%%%%%%%%%%%%%%%%%%%%%%%%%%%%%%%%%%%%%
Throughout this paper, let $q$ be a power of a prime $p$, and let $\alpha\in\mathbb{F}_q^\times$. We begin with a generalized version of Lucas' theorem, whose special case where $q=p$ reduces to the classical Lucas's theorem; related extensions and applications, can be found in
\cite{mestrovic2014}.

\begin{proposition}[A Generalization of Lucas's Theorem]\label{Lucas}
    For non-negative integer $m$, $n$, write
    \[
    m = \sum_{i=0}^{k} m_iq^i\quad \textrm{and} \quad n = \sum_{i=0}^kn_iq^i
    \]
    where $0 \leq m_i,n_i \leq q-1$ for $0\leq i \leq k$. Then
    \[
        \binom  {m}
                {n}
    \equiv
        \prod_{i=0}^k
            \binom  {m_i}
                    {n_i}
        \pmod{p}.
    \]
\end{proposition}

\begin{proof}
The result follows immediately by expressing each $q$-ary digit $m_i$ and $n_i$ in their respective $p$-adic expansions and applying Lucas's Theorem.
\end{proof}

We next prove a simple number-theoretic lemma that will be used later.

\begin{lemma}
    \label{lem: find k}
    If $M$, $N$ are positive integers that satisfy $1 \leq M < N$, then we have a positive integer $k$ such that $\gcd(M + k(N-1), N^2-1) \mid (N-1)$.
\end{lemma}
\begin{proof}
    Write
    \[
    N^2-1=(N-1)(N+1).
    \]
    Let \(q_1,\dots,q_s\) be the distinct odd prime divisors of \(N+1\) that do not divide \(N-1\). Since
    \[
    \gcd(N-1,N+1)\mid 2,
    \]
    every odd prime divisor of \(N+1\) is relatively prime to \(N-1\). Hence \(N-1\) is invertible modulo each \(q_j\).
    
    We will choose \(k\) so that
    \[
    v_2(M+k(N-1))\leq v_2(N-1)\qquad\text{and}\qquad v_{q_j}(M+k(N-1))=0
    \quad
    \text{for all }1\leq j\leq s.
    \]
    These two conditions imply
    \[
    \gcd(M+k(N-1),N^2-1)\mid N-1.
    \]
    For each \(j\), impose the congruence
    \[
    k\equiv (1-M)(N-1)^{-1}\pmod{q_j}.
    \]
    Then
    \[
    M+k(N-1)\equiv 1\pmod{q_j},
    \]
    so \(q_j\nmid M+k(N-1)\).
    
    It remains to impose a congruence modulo \(2\) in order to control the \(2\)-adic valuation. If \(v_2(M)\neq v_2(N-1)\), we impose $k\equiv 1\pmod 2$. Then \(v_2(k)=0\), so
    \[
    v_2(M+k(N-1))=\min\{v_2(M),v_2(k(N-1))\}\leq v_2(N-1).
    \]
    
    If \(v_2(M)=v_2(N-1)\), we instead impose $k\equiv 0\pmod 2$. Then \(v_2(k(N-1))\geq v_2(N-1)+1\), and therefore
    \[
    v_2(M+k(N-1))=v_2(M)=v_2(N-1).
    \]
    
    Thus, in either case, we obtain a system of congruences of the form
    \[
    \begin{cases}
    k\equiv \varepsilon \pmod 2,\\
    k\equiv (1-M)(N-1)^{-1}\pmod{q_j},\qquad 1\leq j\leq s,
    \end{cases}
    \]
    where \(\varepsilon\in\{0,1\}\) is chosen according to the two cases above. The moduli
    \[
    2,q_1,\dots,q_s
    \]
    are pairwise coprime, so the Chinese remainder theorem guarantees that this system has an integer solution $k$. And, this $k$ meets the two conditions.
\end{proof}
Next, we record some properties about Equation~\eqref{functional equation}.

\begin{lemma}
    \label{lem: u^2 + au^-2 = b}
    Suppose $2\nmid q$ and $\alpha \in \mathbb{F}_q^\times$ is a square. Then, for $\beta\in \mathbb{F}_q^\times$, there is $u\in \mathbb{F}_{q^2}$ such that $u^2 + \alpha u^{-2} = \beta.$
\end{lemma}

\begin{proof}
    Let $\alpha = \gamma^2$ for some $\gamma\in \mathbb{F}_q^\times$, and let
    $u = \left(\sqrt{\beta+2\gamma}-\sqrt{\beta-2\gamma}\right)/{2}\in \mathbb{F}_{q^2}^\times.$
    Then, $u^2\neq 0$
    and
    \[
        u^2 + \alpha u ^ {-2} 
        = \frac{\beta - \sqrt{\beta^2 - 4\gamma^2}}{2} + \frac{2\gamma^{2}}{\beta - \sqrt{\beta^2 - 4\gamma^2}}
        = \beta.
    \]
\end{proof}

\begin{lemma}\label{lem: size-of-preimage}
    Suppose $t\in \mathbb{F}_q$ and $\alpha\in \mathbb{F}_q$, we have $1\leq|\phi_{\alpha}^{-1}(\{t\})|\leq 2$.
\end{lemma}

\begin{proof}
    If $u\in\phi_{\alpha}^{-1}(\{t\})$ for some $t \in \mathbb{F}_q$, then $u$ is a root of 
    $x^2 - t x + \alpha$ in $\mathbb{F}_{q^2}$. Thus, $1\leq |\phi_{\alpha}^{-1}(\{t\})| \leq~2$.
\end{proof}

We recall the following result from~\cite{LidlMullenTurnwald1993}*{Lemma 3.5} to describe the elements in $\mathcal{T}_\alpha.$
\begin{lemma}\label{lemma: order of element in T_gamma}
    We have $u\in \mathcal{T}_{\alpha}$ if and only if $u^{q+1} = \alpha$ or $u\in \mathbb{F}_q^\times$.
\end{lemma}
By Lemma~\ref{lemma: order of element in T_gamma}, $\mathcal{T}_\alpha =\mathbb{F}^\times_q\cup \mathcal{S}_\alpha$ with $\mathcal{S}_{\alpha}=\{u\in\mathbb{F}_{q^2}\mid u^{q+1} = \alpha\}$. The next lemma establishes the existence of elements in $\mathcal{S}_\alpha$ with a prescribed order.

\begin{lemma}
    \label{lem: q+1 | ord(u)}
    There is $u\in \mathcal{S}_{\alpha}$ such that $(q+1)\mid \ord(u)$. In particular, if $\alpha = 1$, there is $u\in \mathcal{S}_{\alpha}$ such that $\ord(u) = q+1$.
\end{lemma}

\begin{proof}
    Let $\xi$ be a multiplicative generator of $\mathbb{F}_{q^2}^{\times}$. Then, $\zeta = \xi^{q+1}$ is a multiplicative generator of $\mathbb{F}_{q}^\times$. Let $m$ be a positive integer less than $q-1$ satisfying $\alpha = \zeta^{m}$.
    By Lemma~\ref{lem: find k}, there exists a positive integer $k$ such that $\gcd(m+k(q-1) ,q^2 -1)\mid (q-1)$. Let $u = \xi^{m+k(q-1)}$. 
    Then,
    \[
    u^{q+1} = \xi^{m(q+1)+k(q^2-1)} = \zeta^{m} = \alpha,
    \]
    which implies $u\in \mathcal{S}_{\alpha}$. Furthermore, its multiplicative order is
    \[
    \ord(u) 
    = \frac{q^2-1}{\gcd(m+k(q-1),q^2-1)},
    \]
    which is a multiple of $q+1$.
    
    In the particular case where $\alpha = 1$, any $u\in \mathcal{S}_{\alpha}$ satisfies $u^{q+1} = 1$, which implies $\ord(u) \mid (q+1)$. Combined with the established result, there exists $u\in \mathcal{S}_{\alpha}$ such that $\ord(u) = q+1$.
\end{proof}

%%%%%%%%%%%%%%%%%%%%%%%%%%%%%%%%%%%%%%%%%%
\section{Exact period of Dickson polynomial}
%%%%%%%%%%%%%%%%%%%%%%%%%%%%%%%%%%%%%%%%%%

This section is devoted to proving Theorem~\ref{thm: exact period of Dickson polynomials}, which provides the exact period of Dickson polynomials under specific conditions.

We recall the recurrence relation of Dickson polynomials ~\cite{LidlMullenTurnwald1993}.
\begin{equation}\label{recurssive relation}
D_n(x,\alpha) = x D_{n-1}(x,\alpha) - \alpha D_{n-2}(x,\alpha),
\end{equation}
with initial conditions $D_1(x,\alpha) = x$ and $D_2(x,\alpha) = x^2 - 2\alpha$. The definition can be extended for $n=0$ as $D_{0}(x,\alpha) = 2$.

\begin{proposition}
    \label{nonsquare case, =2 =x}
    We have
    \[
    D_{q^2-1}(x,\alpha) \equiv 2 \pmod{x^q-x}
    \quad \textrm{and}\quad 
    D_{q^2}(x,\alpha) \equiv x \pmod{x^q-x}.
    \]
\end{proposition}

\begin{proof}
    Given $\beta \in \mathbb{F}_q$, we can find $u\in \mathbb{F}^\times_{q^2}$ such that $u + \alpha u ^ {-1} = \beta$. Hence,
    \[
    D_{q^2-1}(\beta,\alpha) = D_{q^2-1}(u+\frac{\alpha}{u}, \alpha) = u^{q^2-1}+\frac{\alpha^{q^2-1}}{u^{q^2-1}}=2
    \]
    and 
    \[
    D_{q^2}(\beta, \alpha) = D_{q^2}(u+\frac{\alpha}{u}, \alpha) = u^{q^2}+\frac{\alpha^{q^2}}{u^{q^2}}=\beta.
    \]
    Since these equalities hold for all $\beta \in \mathbb{F}_q$, the congruences follow.
\end{proof}
    
    \begin{lemma}\label{lem: nonsquare}
        The period of the sequence 
        $[D_{n}(x, \alpha)\pmod{x^q-x}]_{n\geq1}$
        divides $q^2-1$.
    \end{lemma}

    \begin{proof}
    Since Dickson polynomials satisfy the recurrence relation (Equation~\eqref{recurssive relation}) with initial conditions $D_0(x,\alpha)=2$ and $D_1(x,\alpha) = x$,
    the result follows from Proposition~\ref{nonsquare case, =2 =x}.
\end{proof}

%-----------------------------------------------------------%
%                       SQUARE CASE                         %
%-----------------------------------------------------------%
Now we turn to the special case where $2\nmid q$ and $\alpha\in \mathbb{F}_q^\times$ is a square in $\mathbb{F}_q.$

\begin{proposition}
    \label{square case, =2 =x}
    Suppose $2\nmid q$ and $\alpha \in \mathbb{F}_q^\times$ is a square in $\mathbb{F}_q$. Then, we have
    \[
    D_{\frac{q^2-1}{2}}(x,\alpha) \equiv 2 \pmod{x^q-x}
    \quad \textrm{and}\quad 
    D_{\frac{q^2+1}{2}}(x,\alpha) \equiv x \pmod{x^q-x}.
    \]
\end{proposition}

\begin{proof}
    Given $\beta\in \mathbb{F}_q$, by Lemma~\ref{lem: u^2 + au^-2 = b}, there is $u$ in $\mathbb{F}_{q^2}$ such that $u^2 + \alpha u^{-2} = \beta$. Following the same argument as in Proposition~\ref{nonsquare case, =2 =x}, one can verify that
    $D_{\frac{q^2-1}{2}}(\beta,\alpha) = 2$ and $D_{\frac{q^2+1}{2}}(\beta,\alpha) = \beta$, for any $\beta\in\mathbb{F}_q$, completing the proof.
\end{proof}

\begin{lemma}
\label{lem: square}
    Let $\alpha \in \mathbb{F}_q^\times$ be a square in $\mathbb{F}_q$. Then, the period of the sequence 
    divides $\frac{q^2-1}{2}$.
\end{lemma}

\begin{proof}
    Since Dickson polynomials satisfy the recurrence relation (Equation~\eqref{recurssive relation}) with initial conditions $D_0(x,\alpha)=2$ and $D_1(x,\alpha) = x$,
    the result follows from Proposition~\ref{square case, =2 =x}.
\end{proof}

The following proposition demonstrates that all periods are exact.

\begin{proposition}
    \label{prop: exact}
    Let $n$ be a positive integer such that $D_{n}(x, \alpha) \equiv 2\pmod{x^q-x}$. If $2\mid q$, then $(q^2-1)\mid n$. If $2 \nmid q$, then $\frac{q^2-1}{2} \mid n$; especially, if $\alpha$ is nonsquare in $\mathbb{F}_q$, then $q^2-1\mid n$.
\end{proposition}

\begin{proof}
    Fix $\alpha \in \mathbb{F}_q^\times$ and suppose $n$ is a positive integer such that $D_{n}(x,\alpha) \equiv 2 \pmod{x^q-x}$. Then, $D_{n}(x,\alpha)$ satisfies the following commute diagram.
    \[
    \begin{tikzcd}
        \mathcal{T}_{\alpha} \arrow{r}{x^n} \arrow[swap]{d}{\phi_{\alpha}} & (\mathcal{T}_{\alpha})^n \arrow{d}{\phi_{\alpha^n}} \\%
        \mathbb{F}_q^\times \arrow{r}{D_n(x,\alpha)}& \{2\}
    \end{tikzcd}
    \]
    
    We first claim $(q-1) \mid n$.
    By Lemma~\ref{lem: size-of-preimage}, we have $1\leq |(\mathcal{T}_{\alpha})^n|\leq 2$. 
    When $2\mid q$, since $\mathbb{F}_q\subseteq \mathcal{T}_{\alpha}$, we have $(q-1) \mid n$.
    When $2\nmid q$, the same argument gives $\frac{q-1}{2} \mid n$; if the quotient were odd, then $1$ and $-1$ would both be roots of $x^2-2x+\alpha^n$, forcing $2 = 1 + (-1) =0$ in $\mathbb{F}_q$, which contradicts $2\not\mid q$. Thus, $(q-1) \mid n$ in all cases, and 
    $
    \phi_{\alpha^n}^{-1}(\{2\}) = (\mathcal{T}_{\alpha})^n = \{1\}.
    $

    We next claim $(q+1)\mid n$. By Lemma~\ref{lem: q+1 | ord(u)}, there exists $u\in \mathcal{S}_{\alpha}\subseteq \mathcal{T}_{\alpha}$ such that $(q+1)\mid \mathrm{ord}(u)$. Since $u^n\in (\mathcal{T}_{\alpha})^n$ = \{1\}, we have $u^n = 1$, and so $(q+1)\mid n$.

    Combining the two claims, both $q-1$ and $q+1$ divide $n$. When $2 \mid q$, we have $\gcd(q-1,q+1) = 1$ giving $(q^2-1)\mid n$; when $2 \nmid q$, we have $\gcd(q-1,q+1) = 2$, giving $\frac{q^2-1}{2}\mid n$.
    
    In particular, when $2\nmid q$, and $\alpha$ is nonsquare in $\mathbb{F}_{q}^\times$. Let $\xi$ be a generator of $\mathbb{F}_{q^2}^\times$, and set $\zeta = \xi^{q+1}$, which is a generator of $\mathbb{F}_q^\times$, and write $\alpha = \zeta^{m}$ for some odd integer $m$. Take $u = \xi^m \in \mathcal{S}_\alpha$, which is nonsquare in $\mathbb{F}_{q^2}$. If $n\equiv\frac{q^2-1}{2}\pmod{q^2-1}$, we have
    $
    u^{n} = u^{(q^2-1)/2} =  -1 \neq 1,
    $
    contradicting $(\mathcal{T}_\alpha)^n =\{1\}.$
    Therefore, $(q^2-1) \mid n$.
\end{proof}
 
Now, Theorem~\ref{thm: exact period of Dickson polynomials} derives directly from Proposition~\ref{prop: exact}, Lemma~\ref{lem: nonsquare}, and Lemma~\ref{lem: square}.

\begin{proof}[Proof of Theorem \ref{thm: exact period of Dickson polynomials}]
    If $2\nmid q$ and $\alpha$ is a square in $\mathbb{F}_q$,
    combine the results of Proposition~\ref{prop: exact} and Lemma~\ref{lem: square}, we have
    the exact period of the sequence is a factor and also a multiple of $\frac{q^2-1}{2}$. Thus the exact period of the sequence is $\frac{q^2-1}{2}$.

    Similarly, for other cases, combine the results of Proposition~\ref{prop: exact} and Lemma~\ref{lem: nonsquare}, we have
    the exact period of the sequence is a factor and also a multiple of $q^2-1$, and the result follows.
\end{proof}
Theorem~\ref{thm: exact period of Dickson polynomials} yields the following application to functional identity. 
\begin{corollary}\label{cor: symmetric of dickson}
    Let $q$ be a prime power and $\alpha\in \mathbb{F}_q^\times$. Then for $0\leq i \leq q^2-1$ we have
    \[
    D_{q^2-1-i}(x,\alpha) \equiv \alpha^{-i}D_{i}(x,\alpha)\pmod{x^q-x}.
    \]
    Moreover, if $2\nmid q$ and $\alpha\in \mathbb{F}_q^\times$ is a square in $\mathbb{F}_q$, then for $0\leq i \leq \frac{q^2-1}{2}$ we have
    \[
    D_{\frac{q^2-1}{2}-i}(x,\alpha) \equiv \alpha^{-i}D_{i}(x,\alpha)\pmod{x^q-x}.
    \]
\end{corollary}

\begin{proof}
    Notice that the recurrence relation of Dickson polynomials (Equation~\eqref{recurssive relation}) implies
    \begin{equation}\label{reverse recursive relation}
        D_{n-2}(x,\alpha) = \frac{x}{\alpha}D_{n-1}(x,\alpha)-\frac{1}{\alpha}D_{n}(x,\alpha)\textrm{, for $n\geq 3$}.
    \end{equation}
    
    For $i = 0, 1$, it follows from Theorem~\ref{thm: exact period of Dickson polynomials} that 
    \[
    D_{q^2-1}(x,\alpha)\equiv 2 = 2 \alpha^0 \equiv \alpha^0 D_{0}(x,\alpha)\pmod{x^q-x},
    \]
    and 
    \[
    D_{q^2-2}(x,\alpha)
    = \frac{x}{\alpha}D_{q^2-1}(x,\alpha)-\frac{1}{\alpha}D_{q^2}(x,\alpha) \equiv \frac{1}{\alpha}x \equiv \alpha^{-1}D_{1}(x,\alpha)
    \pmod{x^q-x}.
    \]
    Then the first identity follows from induction on $i$, using the recurrence relation~\eqref{reverse recursive relation}, with the initial cases verified above.
    
    Similarly, the second identity can be proved by the square case of Theorem~\ref{thm: exact period of Dickson polynomials}.
\end{proof}

%%%%%%%%%%%%%%%%%%%%%%%%%%%%%%%%%%%%%%%%%%
\section{Dynamics of Dickson polynomials}
%%%%%%%%%%%%%%%%%%%%%%%%%%%%%%%%%%%%%%%%%%
For a periodic sequence $[a_n]_{n\geq 1}$, we say that the sequence has exact tail length $\mathscr l$ and exact period $\mathscr k$ if both integers are the smallest integers such that $a_{\mathscr l+\mathscr k}=a_{\mathscr l}$. We call the pair $(\mathscr l,\mathscr k)$ the dynamical structure of the sequence. In this section, we aim to determine the dynamical structure of the sequence $[D_n(x,\alpha)\pmod{x^q-x}]_{n\geq1}$.

\begin{comment}
    
\sout{Since the period of $[D_n(x)\pmod{x^q-x}]_n$ is $\pi_q(\alpha)$, we can use this information to deduce the dynamic structure of $[D_n^m(x)\pmod{x^q-x}]_m$. First, we observe that if $n$ is coprime to $q^2-1$, meaning there exists an integer $m$ such that $n^m\equiv 1\pmod{\pi_q(\alpha)}$, then we have
\[
D_n^{m+1}(x)=\will{D_{n^{m+1}}(x)}=D_n(x).
\]
Thus, $[D^m_n(x)\pmod{x^q-x}]_m$ is periodic, and the period must divide the order of $n$ in the multiplicative group $(\mathbb{Z}/\pi_q(\alpha)\mathbb{Z})^\times$. To determine the exact period, we must show that there is no positive integer $m'<m$ such that $D_n^{m'}(x)=D_n^m(x)=x$. However, this case can occur in $\mathbb{F}_5[x]$, where
\[
D_{5}(x,1) \equiv x \equiv D_{7}(x,1)\ \pmod{x^5-x}
\]
despite $5\not \equiv 7\pmod{24}$ or $\pmod{12}$. We will address this phenomenon later.}
\sout{A more complex question arises when $n$ is not coprime to $q^2-1$. In one extreme scenario, when a power of $n$ is divisible by $q^2-1$, the $k$-th iteration of $D_n$ becomes a constant function for sufficiently large $k$. However, in most cases, $[n^k\pmod{\pi_q(\alpha)}]_k$ has a dynamic structure $(l,k)$ where both $l$ and $k$ are greater than $1$.}
\end{comment}
We assume $n$ and $\alpha$ satisfy $\alpha^n=\alpha$ in this section. When $n$ is coprime to $q^2-1$, we will explicitly determine the homomorphism $f$, Equation~\eqref{the map}. When $n$ is not coprime to $q^2-1$, we give a partial answer on the dynamical structure.
\subsection{$n$ is coprime to $q^2-1$}
We deal with the case when $n$ is coprime to $q^2-1$. We need some lemmas.
\begin{lemma}\label{lem: D_n(x,1)=kx}
    If $D_n(x,1)=cx\pmod{x^{q}-x}$ for some $c$, then $c=1$.
\end{lemma}
\begin{proof}
    Since $D_n(2,1)=2$, we must have $c=1$ when $q$ is odd. For even $q$, since the coefficients of $D_n(x,1)$ are in $\mathbb{F}_2$, the only possible choice is $c=1$. 
\end{proof}
We recall that $\mathcal{S}_\alpha=\{u\in\mathbb{F}_{q^2}|u^{q+1}=\alpha\}$ for the following proof.
\begin{lemma}\label{lem: kernel}
    Let $n$ be a positive integer such that $D_{n,\alpha}(x)\equiv x \pmod{x^q-x}$. Then, for even $q$, we have:
    \begin{enumerate}
        \item $n^2\equiv 1\pmod{q^2-1}$.
        \item When $\alpha=1$, $n\equiv \pm 1,\ \pm q\pmod{q^2-1}$
        \item When $\alpha\neq 1$, $n \equiv 1,q \pmod{q^2-1}$
    \end{enumerate}
    For odd $q$,
    \begin{enumerate}
        \item $n^2\equiv 1\mod\dfrac{q^2-1}{2}$. In particular, if $\alpha$ is nonsquare, we have $n^2\equiv 1\pmod{q^2-1}$.
        \item When $\alpha=1$, $n\equiv \pm 1,\ \pm q\pmod{\frac{q^2-1}{2}}$
        \item 
        When $\alpha\neq 1$, $n\equiv 1\ \textrm{or}\ q\pmod{\pi_q(\alpha)}$.
        \item We have 
        \[
            \{n|\alpha^n=\alpha\text{ and }D_{n,\alpha}(x)=x\pmod{x^q-x}\}\cong\begin{cases}
            \mathbb{Z}/2\mathbb{Z}, &\text{ if }\alpha\neq 1\\
            \mathbb{Z}/2\mathbb{Z}\oplus\mathbb{Z}/2\mathbb{Z}, &\text{ if }\alpha=1.
            \end{cases}
        \]
    \end{enumerate}
\end{lemma}
\begin{proof}
    \begin{enumerate}
        \item If $D_n(x) \equiv x \pmod{x^q-x}$, then for any $\beta\in \mathbb{F}_q^\times$, $x^n$ maintains or swaps the two solutions of $\phi_\alpha(x)=\beta$. Therefore, $x^{n^2}$ is an identity map on $\mathcal{T}_\alpha$.  
        From Lemma ~\ref{lemma: order of element in T_gamma} and Lemma \ref{lem: q+1 | ord(u)}, 
        there are $u_1,u_2\in \mathcal{T}_\alpha$ such that $q-1 \mid \mathrm{ord}(u_1)$ and $q+1 \mid \mathrm{ord}(u_2)$. We divide the discussion according to the parity of $q$. For even $q$, since $\gcd(q-1,q+1) = 1$, we have 
        \[
        n^2 \equiv 1\  \pmod{q^2-1}.
        \]
        For odd $q$, since $\gcd(q-1,q+1) = 2$, we have
        \[
        n^2 \equiv 1\  \pmod{\frac{q^2-1}{2}}.
        \]
        In particular, if $\alpha$ is nonsquare in $\mathbb{F}_q$, by Lemma ~\ref{lemma: order of element in T_gamma}, there is $u\in \mathcal{T}_\alpha$ such that $u^{q+1} = \alpha$, and 
        \[
        u^{\frac{q^2-1}{2}+1} = \alpha^{\frac{q-1}{2}}u = -u
        \]
        and thus $n^2 \equiv 1\  \pmod{q^2-1}.$
        \item Assuming $\alpha=1$ and letting $\zeta$ be a generator of $\mathbb{F}_q^\times$, we have $\zeta^n=\zeta$ or $\zeta^n=\zeta^{-1}$. This implies $n\equiv \pm 1\pmod{q-1}$. By Lemma~\ref{lemma: order of element in T_gamma}, there exists an element $u$ of order $q+1$. Since $u^n$ is either $u$ or $u^{-1}$, $u^{n-1}=1$ or $u^{n+1}=1$. Thus, $n-1$ or $n+1$ have to congruence to zero mod $q+1$. It follows that we have the following four possible systems of congruence equations
        \begin{align*}
        &
        \begin{cases}
            n\equiv 1 & \pmod{q-1}\\
            n\equiv 1 & \pmod{q+1}
        \end{cases} 
        & 
        \begin{cases}
            n\equiv -1 & \pmod{q-1}\\
            n\equiv 1  & \pmod{q+1}
        \end{cases}\\
        &
        \begin{cases}
            n\equiv 1  & \pmod{q-1}\\
            n\equiv -1 & \pmod{q+1}
        \end{cases} 
        &
        \begin{cases}
            n\equiv -1 & \pmod{q-1}\\
            n\equiv -1 & \pmod{q+1}
        \end{cases}
        \end{align*}
        By the Chinese remainder theorem, we can conclude the following:
        \begin{itemize}
            \item if $q$ is even, $n\equiv \pm 1,\ \pm q \pmod{q^2-1}$.
            \item if $q$ is odd, then $n\equiv \pm 1,\pm q \pmod{\frac{q^2-1}{2}}$.
        \end{itemize}
        \item Suppose $\alpha\neq 1$. First consider $1+\alpha \in \mathbb{F}_q$. Since $\phi_{\alpha}^{-1}(1 + \alpha) = \{1, \alpha\}$, we have
        \[
        1 + \alpha^n = \phi_{\alpha^n}(1^n) = D_{n}(1+\alpha,\alpha) = 1 + \alpha.
        \]
        Thus, $\alpha^{n} = \alpha$, and $\phi_{\alpha^n}(x) = \phi_{\alpha}(x) $. We have the following commute diagram.
        \[
            \begin{tikzcd}
                \mathcal{T}_{\alpha} \arrow{r}{x^n} \arrow[swap]{d}{\phi_{\alpha}} & \mathcal{T}_{\alpha} \arrow{d}{\phi_{\alpha}} \\%
                \mathbb{F}_q \arrow{r}{D_{n,\alpha}(x)}& \mathbb{F}_q
            \end{tikzcd}
        \]
        
        Next, let $\zeta$ be a generator of  $\mathbb{F}_q^\times$ and consider $\zeta + \alpha\zeta^{-1},\zeta^{-1}+\alpha\zeta \in \mathbb{F}_q$. We have
        \[
        \phi_{\alpha}^{-1}(\zeta+\alpha\zeta^{-1}) = \{\zeta,\alpha\zeta^{-1}\}
        \quad \textrm{and} \quad
        \phi_{\alpha}^{-1}(\zeta^{-1}+\alpha\zeta)=\{\zeta^{-1},\alpha\zeta\}.
        \]
        If $\zeta^{n} \neq \zeta$, then $(\zeta^{-1})^n\neq \zeta^{-1}$. It follows that
        \[
        \zeta^{n} = \alpha \zeta^{-1} \quad \textrm{and}\quad \zeta^{-n} = \alpha\zeta,
        \]
        and
        \[
        \zeta^{n+1} = \alpha = \zeta^{-(n+1)}.
        \]
        Hence $\alpha = \alpha^{-1}$. 
        If $q$ is even, then $\alpha = 1$, a contradiction. Suppose $q$ is odd,
        since $\alpha\neq 1$, we have $\alpha = -1$, and $n \equiv \frac{q-3}{2}\pmod{q-1}$.
        If $q\equiv 3 \pmod{4}$, then $n$ is even, and so $(-1)^n = 1 \neq -1$, a contradiction.
        If $q\equiv 1\pmod{4}$, let $z$ b   e a square element in $\mathbb{F}_q$, satisfying $z - z^{-1} \neq 0$. Since $\phi_{-1}^{-1}(z - z^{-1}) = \{z,-z^{-1}\}$, but $z^{n} = z^{\frac{q-1}{2}} z^{-1} =z^{-1}$ which is not $z$ or $-z^{-1}$, also a contradiction.   
        Therefore, $\zeta^n = \zeta$. It follows that $\zeta^{n-1} = 1$ and $n\equiv 1\pmod{q-1}$.
        
        Now, let $u\in \mathcal{S}_{\alpha}$ satisfying $q+1\mid \mathrm{ord}(u)$ and consider $u+\alpha u^{-1}\in \mathbb{F}_q.$ Since $
        \phi^{-1}(u+\alpha u^{-1})= \{u,\alpha u^{-1}\}$,
        we have either $u^{n} = u$ or $u^{n} = \alpha u^{-1}$.\\
        \textbf{Case 1.} If $u^n = u$, then $u^{n-1} = 1$ and $n\equiv 1\pmod{q+1}$.\\
        \textbf{Case 2.} If $u^n = \alpha u^{-1}$, then $u^{n+1} = \alpha = u^{q+1}$. It follows that $n+1\equiv q+1\pmod{k(q+1)}$, for some integer $k$. 
        Therefore, $n\equiv -1\pmod{q+1}$.
        
        Finally, combining all the possible cases, we have the following two possible systems of congruence equations
        \begin{align*}
            \begin{cases}
                n\equiv 1 &\pmod{q-1}\\
                n\equiv 1 &\pmod{q+1}
            \end{cases} 
            \quad \mathrm{and} \quad
            \begin{cases}
                n\equiv 1 &\pmod{q-1}\\
                n\equiv -1&\pmod{q+1}
            \end{cases}
            .
        \end{align*}
        By the Chinese remainder theorem, we have 
        \[
        n\equiv 
        \begin{cases}
            1, q \pmod {\frac{q^2-1}{2}} & \text{if } q \text{ is odd,}\\
            1, q \pmod {q^2-1}           & \text{if } q \text{ is even.}
        \end{cases}
        \]

        Suppose $\alpha$ is nonsquare in $\mathbb{F}_q$. We claim $n\equiv 1\text{ or }q\pmod{q^2-1}$, and by the above conclusion we can consider $n\equiv \frac{q^2-1}{2} + 1\ \textrm{or}\  \frac{q^2-1}{2} + q\ \pmod{q^2-1}$. Take $u\in S_{\alpha}$ such that $u + \alpha u^{-1} \neq 0$, then $u^n = \alpha^{(q-1)/2} u\textrm { or } \alpha^{(q-1)/2}u^q$ and will equal to $ -u\ \textrm{or}\ -\alpha u^{-1}$ respectively, which is not $u$ or $\alpha u ^{-1}$, a contradiction. Therefore
        \[
        n\equiv 1\ \textrm{or}\ q\ \pmod{q^2-1}.
        \]
        \item Immediately follow from (1)-(3).
    \end{enumerate}
\end{proof}

Now, Theorem~\ref{thm: group structure} follows directly.
\begin{proof}[Proof of Theorem~\ref{thm: group structure}]
For $n$ satisfying $\alpha^n=\alpha$, we must have $n-1\equiv 0\pmod r$, where $r$ denoted the order of $\alpha$ in $\mathbb{F}_q^\times$, so it is clear that $n$ is in $U_\alpha$. Then, this theorem is obviously follows from the fact that the map $n\mapsto D_{n,\alpha}(x)\mod(x^q-x)$ is a group homomorphism with the kernel given in Lemma~\ref{lem: kernel}.  
\end{proof}
\begin{corollary}
\begin{enumerate}
    \item Let $\alpha\in\mathbb{F}_q^\times$ satisfy $\alpha^n=\alpha$. The period of $D_{n,\alpha}(x)\pmod{x^q-x}$ is $\delta_{n,\alpha}k$, where $k$ is the multiplicative order of $n$ modulo $\pi_q(\alpha)$ with
    \[
    \delta_{n,\alpha}=
    \begin{cases}
        \frac{1}{2}, &\text{ if } n \text{ and } \alpha \text{ satisfy the conditions we describe below,}\\
        1, &\text{ otherwise. }
    \end{cases}
    \]
    The condition is that 
    when $\alpha = 1$, $-1,\pm q \not \in \{n^i \pmod{q^2-1}| i=1,2,3\dots\}$, and that when $\alpha \neq 1$, $q \not \in \{n^i \pmod{\pi_q(\alpha)} | i=1,2,3\dots\}$.
    \item For $\alpha\in(\mathbb{F}_q)^\times$ with the multiplicative order $m$, let $m=2^{k'}p_{1}^{e'_{1}}\cdots p_{l}^{e'_{l}}$. Let $\pi_q(\alpha)=2^kp_1^{e_1}\cdots p_r^{e_r}q_1^{f_1}\cdots  q_s^{f_s}$ where $p_i$ and $q_i$ are odd primes that divide $q-1$ and $q+1$ respectively. Then there exists $n$ with $\alpha^n=\alpha$ such that $D_{n,\alpha}(x)$ has period 
    $$\lcm\coloneqq
    \begin{cases}
        \lcm(\{2^{k-k'},p_1^{e_1-e'_1},\ldots,p_l^{e_l-e'_l},\phi(p_{l+1}^{e_{l+1}}),\ldots \phi(p_n^{e_n}),\phi(q_1^{f_1}),\ldots \phi(q_s^{f_s})\}) & \text{if }k' > 1\\
        \lcm(\{2^{k-2},p_1^{e_1-e'_1},\ldots,p_l^{e_l-e'_l},\phi(p_{l+1}^{e_{l+1}}),\ldots \phi(p_n^{e_n}),\phi(q_1^{f_1}),\ldots \phi(q_s^{f_s})\}) & \text{if }k'=1\\
        \lcm(\{p_1^{e_1-e'_1},\ldots,p_l^{e_l-e'_l},\phi(p_{l+1}^{e_{l+1}}),\ldots \phi(p_n^{e_n}),\phi(q_1^{f_1}),\ldots \phi(q_s^{f_s})\}) & \text{if }k'=0
    \end{cases}
    $$
    where $\phi$ is the Euler's totient function.
    \item Following the above notations, we have $\alpha$ such that the period of $D_{n,\alpha}(x)\pmod{x^q-x}$ equals the maximum of the multiplicative order of $n$ in $\mathbb{Z}/(q^2-1)\mathbb{Z}$ if some $\phi(p_i^{e_i})$ or $\phi(q_i^{f_i})$ is divisible by $2^k$.
\end{enumerate}
\end{corollary}
\begin{proof}
    \begin{enumerate}
        \item This naturally follows from Lemma~\ref{lem: kernel}.
        \item For each prime $p_i$ and $q_i$ that is not a divisor of $m$, we can find generators $\alpha_i$ and $\beta_i$ of the multiplicative groups $(\mathbb{Z}/p_i^{e_i}\mathbb{Z})^{\times}$ and $(\mathbb{Z}/q_i^{f_i}\mathbb{Z})^{\times}$ respectively. Let us consider an integer $n$ satisfying the following system of congruence equations
        \[
        \begin{cases}
        n\equiv 1+2^{k'}    &\pmod{2^k}\\
        n\equiv 1+p_i^{e'_i}&\pmod{p_i^{e_i}}\quad\text{ for }i\leq l,\\ 
        n\equiv \alpha_i&\pmod{p_i^{e_i}}\quad\text{ for }i>l,\\
        n\equiv \beta_i&\pmod{q_i^{f_i}}.
        \end{cases}
        \]
        Note that the multiplicative order of $n$ in $(\mathbb{Z}/\pi_q(\alpha)\mathbb{Z})^{\times}$ is $\lcm$. Let's replace $\beta_1$ by $\beta_1^2$. Since $2$ divides $\phi(q_i^{f_i})$ for every $q_i$, the solution $n$, after the replacement, has the same multiplicative order. 
        Moreover, since we replace $\beta_1$ by its square, it follows that $-1$ and $q$ is not in the multiplicative subgroup generated by $n$.
        In particular, to exclude the possibility that $-q$ can be generated by $n \pmod{q^2-1}$ when $\alpha =1$, we can further replace $\alpha_1$ by $\alpha_1^2$.

        \item Let $\alpha$ be an element in $\mathbb{F}_q^\times$ of order $2^k$. Clearly, $\alpha$ is nonsquare. If $2^k$ divides either $\phi(p_i^{e_i})$ or $\phi(q_i^{f_i})$, the $\lcm$ are the same with or without the term $2^k$ and is the maximum multiplicative order of elements in the group $\mathbb{Z}/(q^2-1)\mathbb{Z}$.
    \end{enumerate}
\end{proof}
\subsection{$n$ is not coprime to $q^2-1$}
Now, let's discuss the case when $n$ is not coprime to $q^2-1$. Let $(l,k)$ be the dynamical structure of $[n^k\pmod{\pi_q(\alpha)}]_k$, so
\begin{equation}\label{eq: preperiodic}
    D^{l}_n(x)=D_{n^{l}}(x)=D_{n^{l+k}}(x)=D_n^{l+k}(x).
\end{equation}
Equation~(\ref{eq: preperiodic}) implies $\mathscr{l}=l-\mathscr{m_l}\mathscr{k}$ and $k=\mathscr{m_k}\mathscr{k}$ for some integers $\mathscr{m_l}$ and $\mathscr{m_k}$. The integer $\mathscr{k}$ is also the smallest positive integer for which $D_n^\mathscr{k}(x)$ acts as the identity map on the periodic points of $D_n(x)$ in $\mathbb{F}_q$. The tail part of $x^n$ is well-behavior under the map $\phi_\alpha$ thus allows for simple analysis.
\begin{lemma}
    Using the notation defined above, we have $l=\mathscr{l}$.
\end{lemma}
\begin{proof}
    Let $\gamma\in \mathbb{F}_q^\times$, and suppose $\gamma$ is not a periodic point under $x^n$, but $\phi_\alpha(\gamma)$ is a periodic point under $D_n$. Let $\mathscr{k}$ be the period of $[D_{n^m}(x,\alpha)\pmod{x^q-x}]_{m}$. Then, both $\gamma$ and $\gamma^{n^\mathscr{k}}$ map to the same point under $\phi_\alpha$. That is $\phi_{\alpha}^{-1}(\gamma+\frac{\alpha}{\gamma})=\{\gamma,\gamma^{n^\mathrm{k}}\}$. Moreover, we know that $\phi_\alpha^{-1}(\gamma+\frac{\alpha}{\gamma})=\{\gamma, \frac{\alpha}{\gamma}\}$, so we have $\gamma^{n^\mathscr{k}}=\frac{\alpha}{\gamma}$. It implies $(\gamma^{n^\mathscr{k}})^{n^\mathscr{k}}=(\frac{\alpha}{\gamma})^{n^\mathscr{k}}=\frac{\alpha^{n^\mathscr{k}}}{\gamma^{n^{\mathscr{k}}}}=\frac{\alpha}{\alpha/\gamma}=\gamma$. This shows that $\gamma$ is a periodic point under $x^n$, contradicts with the assumption.
\end{proof}

Although we can determine the tail length, calculating the exact period remains challenging. We present the following proposition and raise a question.

\begin{proposition}
    If $k$ is odd, then $(\mathscr{l},\mathscr{k})=(l,k)$.
\end{proposition}
\begin{proof}
    If $D_{n,\alpha}^m(x)=x\pmod{x^q-x}$ and $x^{n^m}\neq x$ on $\mathbb{F}_q$, then $x^{n^m}$ must be either a transposition or the identity of $\phi_\alpha^{-1}(\gamma)$ for $\gamma$, which is a periodic point under $x^n$. Consequently, $x^{n^{2m}}=x$ on $\mathbb{F}_q$, which implies that $k$ must divide $2m$. Note that $k$ cannot divide $m$; otherwise, we would have $x^{n^m}=x$. Therefore, $k$ must be even.
\end{proof}
When $k$ is even, we sometimes observe $\mathscr{k}=k$ and sometimes $\mathscr{k}=\frac{k}{2}$. At present, we do not have a complete criterion distinguishing these two cases. We therefore leave the following question open.
\begin{question}
When $k$ is even, under what conditions do we have $\mathscr{k}=k$ versus $\mathscr{k}=\frac{k}{2}$?
\end{question}
\subsection{Binomial identity from exact periodicity}
The exact period we established above yields various identities, e.g., Corollary~\ref{cor: symmetric of dickson}. We provide other identities that would be complicated to prove using only combinatorial facts. For the remainder of this section, let $s=q-1$ and $t=\lfloor \frac{q}{2}\rfloor$. We introduce the following notation. For non-negative integers $n,j,k$, define
\[
\beta_{n,j,k} = \frac{n}{n-js-k}\binom{n-js-k}{js+k},\quad
\gamma_{n,j,k} = \frac{n}{n-jt-k}\binom{n-jt-k}{jt+k},\quad
\delta_{n,j,k}= (-1)^j\gamma_{n,j,k}.
\]
We can rewrite the Dickson polynomial in Definition~\ref{def: Dickson polynomial} when $n=q^2-1$ as
\[
D_{q^2-1}(x,\alpha) \equiv 
    \displaystyle 2 + \sum_{k=0}^{s-1}\left(\sum_{j=0}^{ t}\beta_{q^2-1,j,k}\right)(-\alpha)^{k} x^{2(q-k-1)}\pmod{x^q-x}
\]
When $2\not\mid q$ and $n=\frac{q^2-1}{2}$, we have
\[
D_{\frac{q^2-1}{2}}(x,\alpha) \equiv 2(-\alpha)^{\frac{q^2-1}{4}} + \sum_{k=0}^{
t-1} \left(\sum_{j=0}^{t}\gamma_{\frac{q^2-1}{2},j,k}(-\alpha)^{jt}\right)(-\alpha)^kx^{q-1-2k}
\pmod{x^q-x}\]
The coefficients in the double sums can be arranged as the tables in Section~5. These coefficients have the following property.
\begin{corollary}
    \label{cor: sum of beta}
    Let $q$ be a power of $2$. 
    Then, for each fixed $k$ with $0\leq k \leq s-1$, we have
    \[
    \sum_{j=0}^{t}\beta_{q^2-1,j,k}\equiv 0\pmod{2}.
    \]
\end{corollary}
\begin{proof}
    The argument for $\beta_{q^2-1,j,k}$ proceeds as before, by fixing $k$, the sum $(\sum_{j}\beta_{q^2-1,j,k})(-\alpha)^k$ equals the coefficient of the term equivalent to $x^{2(q-k-1)}$ in $D_{q^2-1}(x,\alpha) \pmod{x^2-x}$ for $\alpha\in \mathbb{F}_{q}^{\times}$. Since Theorem~\ref{thm: exact period of Dickson polynomials} gives $D_{q^2-1}(x,\alpha)\equiv 0 \pmod{x^q-x}$, all the corresponding coefficients vanished, yielding the result.
\end{proof}
    We note that an analogous identity holds for odd prime powers $q$, where the proof proceeds via a Catalan number identity. We omit the details here as they fall outside the scope of this paper.
\begin{corollary}
    \label{cor: sum of gamma and delta}
    Let $q$ be a power of odd prime $p$. Then for each fixed $k$ with $0\leq k\leq t-1$, we have
    \begin{align*}
        \sum_{j=0}^{t} \gamma_{\frac{q^2-1}{2},j,k} &\equiv 0 \pmod{p}\textrm{,\quad for $q\equiv 1\pmod{4}$},
        \\
        \sum_{j=0}^{t} \delta_{\frac{q^2-1}{2},j,k} &\equiv 0 \pmod{p}\textrm{,\quad for $q\equiv 3\pmod{4}$}.
    \end{align*}
\end{corollary}
\begin{proof}
    The argument of $\gamma_{\frac{q^2-1}{2},j,k}$ and $\delta_{\frac{q^2-1}{2},j,k}$ is similar to Corollary~\ref{cor: sum of beta}, with the fact that
    \[
    (-\alpha)^{jt} = \begin{cases}
        1 & \textrm{, if $q\equiv 1 \pmod {4}$,}\\
        (-1)^j  &\textrm{, if $q\equiv 3\pmod{4}$.}
    \end{cases}
    \]
     for $\alpha\in (\mathbb{F}_q)^2$ and $0\leq j \leq t$
\end{proof}

%%%%%%%%%%%%%%%%%%%%%%%%%%%%%%%%%%%%%%%%%%
\section{Rotation Property}
%%%%%%%%%%%%%%%%%%%%%%%%%%%%%%%%%%%%%%%%%%
    In this section, we present the rotation property of Dickson polynomials when $q$ is an odd prime power. Consider $D_{120}(x, -1), 3D_{120}(x,2) \in \mathbb{F}_{11}[x]$. Tracking the coefficients of $D_{120}(x,-1)$ and $3D_{120}(x,2)$ for the even-degree term $x^n$, where $n$ ranges from $2$ up to $120$, and arranging them in a table with 10 columns, reading from left to right, we have Table \eqref{tab:tableA} and \eqref{tab:tableB}.
    \begin{table}[H]
        \centering
        % Start of first table
        \begin{subtable}{0.48\textwidth}
            \centering
            {\scriptsize
            \begin{NiceMatrixBlock}[auto-columns-width]
            \begin{NiceTabular}{@{}|cccccccccc|@{}}[rules/color=[gray]{0.9},rules/width=1pt]
                \hline
                3 & 6 & 2 & 10 & 1 &
                6 & 0 & 0 & 0  & 0 \\
                8 & 1 & 2 & 8  & 7 &
                4 & 2 & 0 & 0  & 0 \\
                0 & 4 & 6 & 1  & 4 &
                9 & 2 & 1 & 0  & 0 \\
                0 & 0 & 1 & 7  & 3 &
                1 & 5 & 6 & 3  & 0 \\
                0 & 0 & 0 & 7  & 5 &
                10& 7 & 2 & 9  & 10\\
                0 & 0 & 0 & 0  & 2 &
                3 & 6 & 2 & 10 & 1 \\
                \hline
            \end{NiceTabular}
            \end{NiceMatrixBlock}}
            \caption{$D_{120}(x,-1)$}
            \label{tab:tableA}
        \end{subtable}
        \hfill
        % Start of second table
        \begin{subtable}{0.48\textwidth}
            \centering
            {\scriptsize
            \begin{NiceMatrixBlock}[auto-columns-width]
            \begin{NiceTabular}{@{}|cccccccccc|@{}}[rules/color=[gray]{0.9},rules/width=1pt]
                \hline
                1 & 10& 2& 6& 3 &
                2 & 0 & 0& 0& 0 \\ 
                10& 9 & 2& 7& 10&
                5 & 7 & 0& 0& 0 \\ 
                0 & 3 & 6& 5& 1 &
                3 & 7 & 1& 0& 0 \\ 
                0 & 0 & 1& 2& 9 &
                4 & 1 & 6& 4& 0 \\ 
                0 & 0 & 0& 2& 4 &
                7 & 8 & 2& 1& 8 \\ 
                0 & 0 & 0& 0& 6 &
                1 & 10& 2& 6& 3 \\
                \hline
            \end{NiceTabular}
            \end{NiceMatrixBlock}
            }\caption{$3D_{120}(x,2)$}
            \label{tab:tableB}
        \end{subtable}
    
        \caption{Rotate by $180^\circ$}
        \label{tab:combined}
    \end{table}
    
    To prove this rotation property, we examine Table~\ref{tab:3}, which arranges Table~\ref{tab:tableA} into a 11 columns table with the first cell being empty. The proof proceeds in three steps.
    \begin{table}[H]
        \centering
        {\scriptsize
            \begin{NiceMatrixBlock}[auto-columns-width]
            \begin{NiceTabular}{@{}|ccccccccccc|@{}}[rules/color=[gray]{0.9},rules/width=1pt]
                \hline
                  & 3 & 6 & 2 & 10 & 1 & 6 & 0 & 0 & 0 & 0 \\
                8 & 1 & 2 & 8 & 7 & 4 & 2 & 0 & 0 & 0 & 0 \\
                4 & 6 & 1 & 4 & 9 & 2 & 1 & 0 & 0 & 0 & 0 \\
                1 & 7 & 3 & 1 & 5 & 6 & 3 & 0 & 0 & 0 & 0 \\
                7 & 5 &10 & 7 & 2 & 9 & 10& 0 & 0 & 0 & 0 \\
                2 & 3 & 6 & 2 & 10 & 1 \\ 
                \hline
            \end{NiceTabular}
            \end{NiceMatrixBlock}
        }
        \caption{Coefficients of $D_{120}(x,-1)$}
        \label{tab:3}
    \end{table}

    \begin{lemma}
        \label{lem: bino coef equiv1}
        Let $i$ be an integer satisfying $1\leq i \leq \frac{q^2-1}{2}$, then
        \[
                \frac   {q^2-1}
                        {\frac{q^2-1}{2}+i}
                \binom  {\frac{q^2-1}{2}+i}
                        {\frac{q^2-1}{2}-i}
        \equiv
                \frac   {q^2-1}
                        {4(q^2-i)}
                \binom  {q^2-i}
                        {i - 1}
                \left(\frac{1}{16}\right)^{i-1}
        \ \pmod{p},
        \]
    \end{lemma}

    \begin{proof}
    Let
    \[
        u_{j,k} 
    = 
        \frac   {q^2-1}
                {\frac{q^2-1}{2}+jq+k}
        \binom  {\frac{q^2-1}{2}+jq+k}
                {\frac{q^2-1}{2}-jq-k}
    \quad
    \text{and}
    \quad
        v_{j,k}
    =
        \frac   {q^2-1}
                {4(q^2-jq-k)}
        \binom  {q^2-jq-k}
                {jq+k - 1}
        \left(\frac{1}{16}\right)^{j+k-1}.
    \]
    Then, equivalently we may show
    \[
        u_{j,k} \equiv v_{j,k}
        \pmod{p},
    \]
    for integers $j,\ k$ satisfying $1 \leq jq + k \leq \frac{q^2-1}{2}$, $0\leq j\leq \frac{q-1}{2}$ and $0\leq k\leq q-1$.
     
    \textbf{Case 1.} $j=0$.\\
    \[
    u_{0,k} = \frac{q^2-1}{\frac{q^2-1}{2}+k} \binom{\frac{q^2-1}{2}+k}{\frac{q^2-1}{2} - k}
    \quad
    \textrm{and}
    \quad
    v_{0,k} = \frac{q^2-1}{4(q^2-k)} \binom{q^2-k}{k-1}\left( \frac{1}{16}\right)^{k-1}.
    \]
    For $k=1$, a direct computation shows that
    $u_{0,1}\equiv \frac{1}{4}\equiv v_{0,1}\pmod{p}.$
    Furthermore, since
    \[
    u_{0,k} \equiv
            \frac   {\frac{q^2-1}{2}+1}
                    {\frac{q^2-1}{2}+k}
            \ 
            \frac   {(\frac{q^2-1}{2} + k)^{\underline{k-1}}\ (\frac{q^2-1}{2}-k+1)^{\overline{k-1}}\ 2!}
                    {(2k)!}
            \ 
            u_{0,1}
            \equiv
            \frac   {(-1)^{k-1}}
                    {2^{4k-4}}
            \ 
            \frac   {1}
                    {k}
            \ 
            \binom  {2k-2}
                    {k-1}
            \ 
            u_{0,1}\pmod{p},
    \]
    and
    \[
    v_{0,k} \equiv 
            \frac   {q^2-1}
                    {q^2-k}
            \ 
            \frac   {1}
                    {2^{4(k-1)}}
            \ 
            \frac{(q^2-k)^{\underline{k-1}}}{(k-1)!}
            \ 
            v_{0,1}
            \equiv
            \frac   {(-1)^{k-1}}
                    {2^{4k-4}}
            \ 
            \frac   {1}
                    {k}
            \ 
            \binom  {2k-2}
                    {k-1}
            \ 
            v_{0,1}
            \pmod{p},
    \]
    we have
    $u_{0,k}\equiv v_{0,k}\pmod{p},$
    for $0 \leq k \leq q-1$.
    
    \textbf{Case 2.} $k=0$.
    We first simplify $u_{j,0}$ and $v_{j,0}$ by Proposition \ref{Lucas}
    \[
        u_{j,0} 
    = 
        \frac   {q^2-1}
                {\frac{q^2-1}{2} + jq}
        \binom  {\frac{q^2-1}{2} + jq}
                {\frac{q^2-1}{2} - jq}
    \equiv 
        2
        \binom  {\frac{q-1}{2} + j}
                {\frac{q-1}{2} - j}
    \pmod{p},
    \]
    and 
    \[
    v_{j,0} =
            \frac   {q^2-1}
                    {4(q^2-jq)}
            \binom  {q^2-jq}
                    {jq- 1}
            \left(\frac{1}{16}\right)^{j-1}
        \equiv 
            \frac   {-1}
                    {4}
            \binom  {q - j - 1}
                    {j-1}
            \left(\frac{1}{16}\right)^{j-1}   
        \pmod{p}.
    \]
    Similar to case 1, for $j= 1$, we have $u_{1,0}\equiv -\frac{1}{4} \equiv v_{1,0}\pmod{p}$, and by expressing both $u_{j,0}$ and $v_{j,0}$ as multiples of their respective base cases $u_{1,0}$ and $v_{1,0}$, a straightforward algebraic simplification yields $u_{j,0} \equiv v_{j,0}\ \pmod{p}$, for $1\leq j \leq \frac{q-1}{2}.$

    \textbf{Case 3.} General case. For the rest of the cases, we express both $u_{j,k}$ and $v_{j,k}$ as multiples of their respective base cases $u_{j,0}$ and $v_{j,0}$. That is,
    \begin{align*}
        u_{j,k}
    &=
                \frac   {\frac{q^2-1}{2} +jq}
                    {\frac{q^2-1}{2} + jq + k}
            \ 
            \frac   {(\frac{q^2-1}{2} +jq +k)^{\underline{k}}
                    \ 
                    (\frac{q^2-1}{2} -jq -k+1)^{\overline{k}}}
                    {(2jq+2k)^{\underline{2k}}}
            \ 
            u_{j,0}
            \\
    &\equiv
            \frac   {(-1)^k}
                    {2^{4k-1}\ (2k-1)}
            \ 
            \binom{2k-1}{k} u_{j,0}
            \pmod{p},
    \end{align*}
    and
    \[
            v_{j,k}
        \equiv
            \frac   {1}
                    {4k}
            \ 
            \binom  {q-j-1}
                    {j}
            \ 
            \binom  {q-k}
                    {k-1}
            \ 
            \left(
            \frac   {1}
                    {16}
            \right)^{j+k-1}
        \equiv
            \frac   {(-1)^{k}}
                    {2^{4k-1}(2k-1)}
            \binom  {2k-1}
                    {k}
            \ 
            v_{j,0}
            \pmod{p}.
        \]
    Therefore, it follows from Case 2 that
    \[
        u_{j,k} \equiv v_{j,k}\ \pmod{p}
    \]
    for $1\leq j \leq \frac{q-1}{2}$ and $0\leq k\leq q-1$. 
    This completes the proof.
\end{proof}

    Building on this lemma, we establish a symmetric relation, the first part of Theorem~\ref{main theorem: 1}, in Dickson polynomials that explains the patterns observed in the tables above.
    We recall the explicit expression of Dickson polynomials
    (also see~\cite{Lidl}.)
\begin{definition}\label{def: Dickson polynomial}
For positive integers $n$ and $\alpha \in \mathbb{F}_q$, the Dickson polynomials (of the first kind) over $\mathbb{F}_q$ are given by
\[
D_{n}(x,\alpha)=
\sum_{i=0}^{\lfloor {\frac{n}{2}} \rfloor}
\frac   {n}
{n-i}
\binom  {n-i}
{i}
(-\alpha)^i x^{n-2i}.
\]
\end{definition}

\begin{theorem}
    Let $\alpha \in \mathbb{F}_q^\times$, and we have
    \[
        x^{q^2+1}\left(D_{q^2-1}\left(x^{-1},\alpha\right) - 2\right)
    = 
        (-4\alpha)^{-1} \left(D_{q^2-1}\left(x,(16\alpha)^{-1} \right)-2\right).
    \]
\end{theorem}
\begin{proof}
    For $n = q^2-1$, the constant term ($i=n/2$) of $D_{n}(x,\alpha)$ is $2(-\alpha)^{(q^2-1)/2} = 2$.
    Expanding the left-hand side and applying the index transformation $j=\frac{q^2-1}{2}-i$, we obtain
    \[
        x^{q^2+1}\left(D_{q^2-1}\left(x^{-1},\alpha\right) - 2\right)
    =
        \sum_{j=1}^{\frac{q^2-1}{2}}
            \frac   {q^2-1}
                    {\frac{q^2-1}{2}+j}
            \binom  {\frac{q^2-1}{2}+j}
                    {\frac{q^2-1}{2}-j}
            (-\alpha)^{-j}
            x^{(q^2+1)-2j}.
    \]
    Similarly, for the right-hand side, apply the index transformation $j = i+1$, we have
    \[
            (-4\alpha)^{-1} \left(D_{q^2-1}\left(x,(16\alpha)^{-1} \right)-2\right)
        =
            \sum_{j=1}^{\frac{q^2-1}{2}}
                \frac   {q^2-1}
                        {4(q^2-j)}
                \binom  {q^2-j}
                        {j - 1}
                \left(\frac{1}{16}\right)^ {j - 1}
                (-\alpha)^{-j}
                x^{(q^2+1)-2j}.
    \]
    The result then follows directly by comparing the coefficients of $x^{q^2+1-2j}$ using Lemma 3.1.
\end{proof}

Next, consider the Dickson polynomials $D_{60}(x,1)$, $D_{61}(x,9) \in \mathbb{F}_{11}[x]$. Arranged in five columns,  Table~\ref{tab:tableC} presents the coefficients of the even-degree terms $x^n$ of $D_{60}(x,1)$ where $n$ ranges from $60$ down to $2$, and Table~\ref{tab:tableD} lists the coefficients of the odd-degree terms $x^n$ of $D_{61}(x,9)$ where $n$ ranges from $59$ down to $1$. Notably, Table~\ref{tab:tableC} is found to be a $180^\circ$ rotation of Table~\ref{tab:tableD}. The following lemma explains this relation.
    \begin{table}[H]
        \centering
    
        % Start of first table
        \begin{subtable}{0.48\textwidth}
            \centering
            {\scriptsize
            \begin{NiceMatrixBlock}[auto-columns-width]
            \begin{NiceTabular}{@{}|ccccc|@{}}[rules/color=[gray]{0.9},rules/width=1pt]
                \hline
                1  & 6  & 5  & 0  & 0\\
                10 & 6  & 2  & 2  & 0\\
                0  & 7  & 9  & 2  & 0\\
                0  & 3  & 4  & 5  & 5\\
                0  & 0  & 3  & 7  & 4\\
                0  & 0  & 10 & 6  & 2\\
                \hline
            \end{NiceTabular}
            \end{NiceMatrixBlock}}
            \caption{$D_{60}(x,1)$}
            \label{tab:tableC}
        \end{subtable}
        \hfill
        % Start of second table
        \begin{subtable}{0.48\textwidth}
            \centering
            {\scriptsize
            \begin{NiceMatrixBlock}[auto-columns-width]
            \begin{NiceTabular}{@{}|ccccc|@{}}[rules/color=[gray]{0.9},rules/width=1pt]
                \hline
                2  & 6  & 10 & 0  & 0\\
                4  & 7  & 3  & 0  & 0\\
                5  & 5  & 4  & 3  & 0\\
                0  & 2  & 9  & 7  & 0\\
                0  & 2  & 2  & 6  & 10\\
                0  & 0  & 5  & 6  & 1 \\
                \hline
            \end{NiceTabular}
            \end{NiceMatrixBlock}
            }\caption{$2D_{61}(x,9)$}
            \label{tab:tableD}
        \end{subtable}
    
        \caption{Rotate by $180^\circ$}
        \label{tab:combined2}
    \end{table}

\begin{lemma}
    \label{lem: bino coef equiv2}
    Let $i$ be an integer satisfying $0\leq i\leq \frac{q^2-1}{4},$ then 
    \[
            \frac   {\frac{q^2-1}{2}}
                    {\frac{q^2-1}{4} +i}
            \binom  {\frac{q^2-1}{4} +i}
                    {\frac{q^2-1}{4} -i}
        \equiv 
            \frac   {q^2+1}
                    {\frac{q^2+1}{2} - i}
            \binom  {\frac{q^2+1}{2} - i}
                    {i}
            \left(\frac{1}{16}\right)^{i}
        \pmod{p},
    \]
\end{lemma}

\begin{proof}
    The proof proceeds similarly to Lemma 3.1. After defining $u_{j,k}$ and $v_{j,k}$, we conduct a case analysis based on the residue of $q\pmod{4}$. Verification is performed in three stages:
    \begin{itemize}
        \item[1.] Initial case: First, we verify that $u_{0,0}\equiv 2 \equiv v_{0,0}\ \pmod{p}.$
        \item[2.] Row-wise extension: We then extend the result to show $u_{j,0}\equiv v_{j,0}\ \pmod{p}$ for all relevant $j$.
        \item[3.] General identity: Finally, we derive the general congruence $u_{j,k} \equiv v_{j,k}\ (\mathrm
        {mod}\ p)$ from case 2.
    \end{itemize}
    Detailed calculations for each case follow the same combinatorial logic as in Lemma~\ref{lem: bino coef equiv1} and are thus omitted here for brevity.
\end{proof}

    In the form of Dickson polynomials, we have the following identity, which shows the second part of Theorem~\ref{main theorem: 1}.
    
\begin{theorem}%Rotation Property 2
    Let $\alpha \in\mathbb{F}_q^\times$ be a square, and we have
    \[
    x^{\frac{q^2+1}{2}}D_{\frac{q^2-1}{2}}(x^{-1},\alpha)
    =
    2D_{\frac{q^2+1}{2}}\left(x, (16\alpha)^{-1}\right).
    \]
\end{theorem}
\begin{proof}
    Since $\alpha \in \mathbb{F}_q^\times$ is square and $8 \mid q^2-1$, by applying the index tranformation
    \[
            x^{\frac{q^2+1}{2}}D_{\frac{q^2-1}{2}}(x^{-1},\alpha)
        =
            \sum _{i=0}^{\frac{q^2-1}{4}}
            \frac   {\frac{q^2-1}{2}}
                    {\frac{q^2-1}{4} +i}
            \binom  {\frac{q^2-1}{4} +i}
                    {\frac{q^2-1}{4} -i}
            (-\alpha)^{-i}x^{\frac{q^2+1}{2} - 2i},
    \]
and
    \[
            2D_{\frac{q^2+1}{2}}\left(x,(16\alpha)^{-1}\right)
        =
            2\sum_{i=0}^{\frac{q^2-1}{4}}
            \frac   {\frac{q^2+1}{2}}
                    {\frac{q^2+1}{2} - i}
            \binom  {\frac{q^2+1}{2} - i}
                    {i}
            \left(\frac{1}{16}\right)^{i}
            (-\alpha)^{-i}
            x ^ {\frac{q^2+1}{2} - 2i}.
    \]
    The result follows directly by comparing the coefficients of $x ^ {\frac{q^2+1}{2} - 2i}$ from Lemma~\ref{lem: bino coef equiv2}.
\end{proof}

\appendix
\section{Algorithms}
This section is devoted to the development of algoritms for identifying Dickson polynomials modulo $x^q-x$. For an odd prime power $q$, given an arbitrary polynomial $f(x)\in \mathbb{F}_q[x]$, a straightforward approach is to exhaustively check all candidates. A better method can use the parity symmetric properties of Dickson polynomials. Observe the following second-order recurrence relation.

\begin{proposition}
    \label{prop: recur_halving}
    For $n\geq 4$, the Dickson polynomials satisfy:
    \[
    D_n(x,\alpha) = (x^2-2\alpha) D_{n-2}(x,\alpha) - \alpha^2 D_{n-4}(x,\alpha).
    \]
\end{proposition}

\begin{proof}
    This follows immediately by applying recurrence (Equation~\eqref{recurssive relation}) twice.
\end{proof}

For odd prime power $q$, The Dickson polynomials exhibit a clear parity symmetry: if $n$ is odd, all coefficients of even-degree terms are zero; conversely, if $n$ is even, all coefficients of odd-degree terms vanish. Consequently, before executing a full check, we can immediately identify a polynomial as non-Dickson if its coefficient structure violates this parity.
\begin{comment}
\begin{algorithm}[H]
    \caption{Brute-force Dickson Polynomial Test for Odd Characteristic}
    \label{Alg: brute-force}
\begin{algorithmic}[1]
    \Require{$f(x) \in \mathbb{F}_q[x]$, where $q$ is an odd prime power.}
    \Ensure{Whether $f(x) \equiv D_n(x, \alpha) \pmod{x^q-x}$.}
    \State $g(x) \gets f(x) \pmod{x^q-x}$
    \If{$g(x)$ violates the parity symmetry of Dickson polynomials} \label{line:parity}
        \State \Return Not a Dickson polynomial.
    \EndIf
    \State{start $\gets 1$ \textbf{if} $\deg(g(x))$ is odd, \textbf{else} start $\gets 2$ }
    \For{$\alpha \in \mathbb{F}_q^\times$}
        \State Compute $D_n(x, \alpha) \pmod{x^q-x}$ using the recurrence in Proposition~\ref{prop: recur_halving}.
        \If{$g(x) = D_n(x, \alpha)$ for some $n \le q^2-1$}
            \State \Return $f(x) \equiv D_n(x, \alpha)\pmod{x^q-x}$
        \EndIf
    \EndFor
    \State \Return Not a Dickson polynomial.
\end{algorithmic}
\end{algorithm}
\end{comment}
However, for large $q$, this approach still has $O(q \cdot \pi_q(\alpha))=O(q^3)$ complexity. Alternatively, we can identify the parameter $\alpha$ directly from the coefficients of $f(x)$, thereby avoiding a brute-force search over $\mathbb{F}_q^\times$. This is motivatived by the tables in Section~5.

We then define some pre-computable constants:

\[
\beta_{n,j,k}=
\begin{cases}
    \frac{n}{n-js-k}\binom{n-js-k}{js+k},&\text{if $q$ is even;}\\
    \frac{n}{n-js -k}\binom{n-js -k}{js +k},&\text{if $q$ is odd and $-\alpha$ is square in $\mathbb F_q$;}\\
    (-1)^j \frac{n}{n-js -k}\binom{n-js -k}{js +k},&\text{if $q$ is odd and $-\alpha$ is nonsquare in $\mathbb F_q$.}
\end{cases}
\]

\[
s=\begin{cases}
    q-1,&\text{$q$ is even;}\\
    \frac{q-1}{2},&\text{$q$ is odd.}
\end{cases}\qquad
B=\begin{cases}
    
    2(-\alpha)^{n/2},&\text{if $q$ is odd and $n$ is even;}\\
    0,&\text{otherwise.}
\end{cases}
\]
and define
\[
B_{n,k,\alpha} =\sum_{\substack{j\geq 0 \\js+k<n/2}}\beta_{n,j,k}.
\]
\begin{comment}
For even $q$, let $s=q-1$, and $\beta_{n,j,k} = \frac{n}{n-js-k}\binom{n-js-k}{js+k}$, and define
\[
B_{n,k,\alpha} =\sum_{\substack{j\geq 0 \\js+k<n/2}}\beta_{n,j,k}.
\]
For odd $q$, let $s = \frac{q-1}{2}$, $\gamma_{n,j,k} = \frac{n}{n-js -k}\binom{n-js -k}{js +k}$ and $\delta_{n,j,k}=(-1)^j \frac{n}{n-js -k}\binom{n-js -k}{js +k}$, and define
\[
B_{n,k,\alpha}= \begin{cases}
    \displaystyle
     \sum_{\substack{j\geq 0\\
    js+k<n/2}
    }\gamma_{n,j,k} &\textrm{, if $-\alpha$ is square in $\mathbb{F}_q$,}
    \\ \\
    \displaystyle
    \sum_{\substack{j\geq 0\\
    js+k<n/2}} \delta_{n,j,k}&\textrm{, if $-\alpha$ is nonsquare in $\mathbb{F}_q$.}\\ \\
    2(-\alpha)^{n/2}&\textrm{, if $k = n/2.$}
\end{cases}
\]
\end{comment}
Then Dickson polynomial can be uniformly express as
\[
D_n(x,\alpha)\equiv B + \sum_{k=0}^{s-1} B_{n,k,\alpha}\ (-\alpha)^k x^{n+2(s-k)}\pmod{x^q-x},
\]
Thus, whenever \(B_k\neq 0\), the coefficient of the term corresponding to \(x^{n-2k}\) determines \(\alpha^k\). In particular, if \(B_{n,1,\alpha}\neq 0\), then \(\alpha\) can be recovered directly by dividing the coefficient of the term \(x^{n-2}\) in \(D_n(x,\alpha)\pmod{x^q-x}\) by \(B_{n,1,\alpha}\).

More generally, suppose $f(x)\equiv D_n(x,\alpha)\pmod{x^q-x}$ for an unkown nonzero $\alpha$. Let $C_k$ denote the coefficient of the term corresponding to $x^{n-2k}$ in $f(x)$. If $B_{n,k,\alpha}\neq 0$, then we must have $C_k\neq 0$; otherwise we may immediately conclude that $f(x)$ is not a Dickson polynomial.

Let $S=\{k|C_k\neq 0\}$ and let $d=\gcd(\{q-1\}\cup S)$. Then there exist integers $c_k$ such that
\[
d\equiv \sum_{k\in S} c_kk\pmod{q-1}
\]
Thus, we have 
\[
(-\alpha)^d = \prod_{k\in S}\left(\frac{C_k}{B_{n,k,\alpha}}\right)^{c_k}.
\]
If \(d=1\), this determines \(\alpha\) uniquely. If \(d>1\), then the possible values of \(\alpha\) form a set of at most \(d\) candidates, which can then be checked directly. The resulting procedure is summarized in the following pseudocode. It reduces the search over \(\alpha\) substantially once \(n\) is fixed. The remaining bottleneck is the determination of the index \(n\), which resembles a discrete-logarithm-type problem arising from the power-map structure behind Dickson polynomials.

\begin{algorithm}[H]
    \caption{Parameter $\alpha$ Extraction}
\begin{algorithmic}[1]
    \Require{$f(x)\in\mathbb{F}_q[x]$.}
    \Ensure{Whether $f(x) \equiv D_n(x,\alpha)\pmod{x^q-x}$ for some $n$ and $\alpha \in \mathbb{F}_q^\times$.}
    \State Check if $f(x)$ fails for symmetric test or a monomial.
    \For{$n = 1$ \textbf{to} $q^2 - 1$}
        \State Compute $B_k$
        \State $S=\{k|C_k\neq 0\}$
        \State $d=\gcd(q-1, S)$
        \State Compute $c_k$ by B\'{e}zout identity
        \If{d==1}
            \State Uncover the unique $\alpha$
        \Else
            \State Try at most $d$ many candidates of $\alpha$
        \EndIf
    \EndFor
\end{algorithmic}
\end{algorithm}

\section*{Acknowledgement}
We would like to express our sincere gratitude to Wei-Hsuan Yu, Shen-Fu Tsai, B. Hutz, T. Tucker, M. Manes, and B. Thompson for their invaluable support and insightful discussions throughout this work. Peng would also like to thank Howard Yu-Hao Chen for him pointing out a serious typo in associate names. Chen would also like to thank Sen-Peng Eu for his guidance and encouragement. Finally, we both wish to thank Liang-Chung Hsia for his continuous support and collaboration. Finally, We thank the reviewers for excellent suggestions, which significantly improved the clarity of our manuscript.

\bibliographystyle{amsplain}
\bibliography{reference}

@book{mullen2013handbook,
  title = {Handbook of Finite Fields},
  author = {Mullen, G.L. and Panario, D.},
  isbn = {9781439873823},
  lccn = {2013013450},
  series = {Discrete Mathematics and Its Applications},
  url = {https://books.google.com.tw/books?id=YADSBQAAQBAJ},
  year = {2013},
  publisher = {CRC Press}
}

@book{lidl1997finite,
  title = {Finite Fields},
  author = {Lidl, R. and Niederreiter, H.},
  number = {v. 20, pt. 1},
  isbn = {9780521392310},
  lccn = {96031467},
  series = {EBL-Schweitzer},
  url = {https://books.google.com.tw/books?id=xqMqxQTFUkMC},
  year = {1997},
  publisher = {Cambridge University Press}
}

@article{WANG2012814,
  author = {Wang, Qiang and Yucas, Joseph L.},
  title = {Dickson polynomials over finite fields},
  journal = {Finite Fields and Their Applications},
  volume = {18},
  number = {4},
  pages = {814--831},
  year = {2012},
  doi = {10.1016/j.ffa.2012.02.001},
  url = {https://www.sciencedirect.com/science/article/pii/S1071579712000214}
}

@article{Hou2015,
  author = {Hou, Xiang-dong},
  title = {Permutation polynomials over finite fields---a survey of recent advances},
  journal = {Finite Fields and their Applications},
  volume = {32},
  year = {2015},
  pages = {82--119},
  doi = {10.1016/j.ffa.2014.10.001},
  url = {https://doi.org/10.1016/j.ffa.2014.10.001}
}

@incollection{Lidl,
  author = {Lidl, Rudolf},
  title = {Theory and Applications of Dickson Polynomials},
  booktitle = {Topics in Polynomials of One and Several Variables and Their Applications},
  pages = {371--395},
  doi = {10.1142/9789814360296_0023},
  url = {https://www.worldscientific.com/doi/abs/10.1142/9789814360296_0023}
}

@book{LidlMullenTurnwald1993,
  author = {Lidl, R. and Mullen, G. L. and Turnwald, G.},
  title = {Dickson Polynomials},
  series = {Pitman Monographs and Surveys in Pure and Applied Mathematics},
  volume = {65},
  publisher = {Longman Scientific \& Technical},
  year = {1993},
  isbn = {0-582-09119-5}
}

@article{ChouGomezMullen1988,
  author = {Chou, Wun Seng and Gomez-Calderon, Javier and Mullen, Gary L.},
  title = {Value sets of Dickson polynomials over finite fields},
  journal = {Journal of Number Theory},
  volume = {30},
  number = {3},
  year = {1988},
  pages = {334--344},
  doi = {10.1016/0022-314X(88)90006-6},
  url = {https://doi.org/10.1016/0022-314X(88)90006-6}
}

@article{YuanDing2014,
  author = {Yuan, Pingzhi and Ding, Cunsheng},
  title = {Further results on permutation polynomials over finite fields},
  journal = {Finite Fields and their Applications},
  volume = {27},
  year = {2014},
  pages = {88--103},
  doi = {10.1016/j.ffa.2014.01.006},
  url = {https://doi.org/10.1016/j.ffa.2014.01.006}
}

@article{ALDENGASSERT201483,
title = {Chebyshev action on finite fields},
journal = {Discrete Mathematics},
volume = {315-316},
pages = {83-94},
year = {2014},
issn = {0012-365X},
doi = {https://doi.org/10.1016/j.disc.2013.10.014},
url = {https://www.sciencedirect.com/science/article/pii/S0012365X1300441X},
author = {T. {Alden Gassert}},
}

@incollection {MartinsPanarioQureshiSurvey,
    AUTHOR = {Martins, Rodrigo and Panario, Daniel and Qureshi, Claudio},
     TITLE = {A survey on iterations of mappings over finite fields},
 BOOKTITLE = {Combinatorics and finite fields---difference sets,
              polynomials, pseudorandomness and applications},
    SERIES = {Radon Ser. Comput. Appl. Math.},
    VOLUME = {23},
     PAGES = {135--172},
 PUBLISHER = {De Gruyter, Berlin},
      YEAR = {[2019] \copyright 2019},
      ISBN = {978-3-11-064179-0; 978-3-11-064209-4; 978-3-11-064196-7},
   MRCLASS = {11T71 (11T06 37P05 94A62)},
  MRNUMBER = {4359974},
       DOI = {10.1515/9783110642094-008},
       URL = {https://doi.org/10.1515/9783110642094-008},
}

@article{QureshiPanarioChebyshev,
    AUTHOR = {Qureshi, Claudio and Panario, Daniel},
     TITLE = {The graph structure of {C}hebyshev polynomials over finite
              fields and applications},
   JOURNAL = {Des. Codes Cryptogr.},
  FJOURNAL = {Designs, Codes and Cryptography. An International Journal},
    VOLUME = {87},
      YEAR = {2019},
    NUMBER = {2-3},
     PAGES = {393--416},
      ISSN = {0925-1022,1573-7586},
   MRCLASS = {12E20 (11T06 11T71)},
  MRNUMBER = {3911212},
MRREVIEWER = {Fengwei\ Li},
       DOI = {10.1007/s10623-018-0545-7},
       URL = {https://doi.org/10.1007/s10623-018-0545-7},
}

@article{QureshiPanarioRedei,
    AUTHOR = {Qureshi, Claudio and Panario, Daniel},
     TITLE = {R\'edei actions on finite fields and multiplication map in
              cyclic group},
   JOURNAL = {SIAM J. Discrete Math.},
  FJOURNAL = {SIAM Journal on Discrete Mathematics},
    VOLUME = {29},
      YEAR = {2015},
    NUMBER = {3},
     PAGES = {1486--1503},
      ISSN = {0895-4801,1095-7146},
   MRCLASS = {05C90 (11T06)},
  MRNUMBER = {3384830},
MRREVIEWER = {M.\ Stef\u anescu},
       DOI = {10.1137/140993338},
       URL = {https://doi.org/10.1137/140993338},
}

@article {QureshiPanarioMartinsRedei,
    AUTHOR = {Qureshi, Claudio and Panario, Daniel and Martins, Rodrigo},
     TITLE = {Cycle structure of iterating {R}edei functions},
   JOURNAL = {Adv. Math. Commun.},
  FJOURNAL = {Advances in Mathematics of Communications},
    VOLUME = {11},
      YEAR = {2017},
    NUMBER = {2},
     PAGES = {397--407},
      ISSN = {1930-5346,1930-5338},
   MRCLASS = {11T71 (94A62)},
  MRNUMBER = {3651311},
MRREVIEWER = {Intan\ Muchtadi-Alamsyah},
       DOI = {10.3934/amc.2017034},
       URL = {https://doi.org/10.3934/amc.2017034},
}

@article{VasigaShallit,
    AUTHOR = {Vasiga, Troy and Shallit, Jeffrey},
     TITLE = {On the iteration of certain quadratic maps over {${\rm
              GF}(p)$}},
   JOURNAL = {Discrete Math.},
  FJOURNAL = {Discrete Mathematics},
    VOLUME = {277},
      YEAR = {2004},
    NUMBER = {1-3},
     PAGES = {219--240},
      ISSN = {0012-365X,1872-681X},
   MRCLASS = {05C20 (11T99)},
  MRNUMBER = {2033734},
MRREVIEWER = {Kevin\ Lawrence\ McAvaney},
       DOI = {10.1016/S0012-365X(03)00158-4},
       URL = {https://doi.org/10.1016/S0012-365X(03)00158-4},
}

@article{MartinsPanarioRandomMappings,
    AUTHOR = {Martins, Rodrigo S. V. and Panario, Daniel},
     TITLE = {On the heuristic of approximating polynomials over finite
              fields by random mappings},
   JOURNAL = {Int. J. Number Theory},
  FJOURNAL = {International Journal of Number Theory},
    VOLUME = {12},
      YEAR = {2016},
    NUMBER = {7},
     PAGES = {1987--2016},
      ISSN = {1793-0421,1793-7310},
   MRCLASS = {12Y05 (11T06 37P05)},
  MRNUMBER = {3544423},
MRREVIEWER = {Simone\ Ugolini},
       DOI = {10.1142/S1793042116501219},
       URL = {https://doi.org/10.1142/S1793042116501219},
}

@article{BachPollardRho,
    AUTHOR = {Bach, Eric},
     TITLE = {Toward a theory of {P}ollard's rho method},
   JOURNAL = {Inform. and Comput.},
  FJOURNAL = {Information and Computation},
    VOLUME = {90},
      YEAR = {1991},
    NUMBER = {2},
     PAGES = {139--155},
      ISSN = {0890-5401,1090-2651},
   MRCLASS = {11Y05 (11A51)},
  MRNUMBER = {1094034},
MRREVIEWER = {Maurice\ Mignotte},
       DOI = {10.1016/0890-5401(91)90001-I},
       URL = {https://doi.org/10.1016/0890-5401(91)90001-I},
}

@article{BrentPollardFermat,
    AUTHOR = {Brent, Richard P. and Pollard, John M.},
     TITLE = {Factorization of the eighth {F}ermat number},
   JOURNAL = {Math. Comp.},
  FJOURNAL = {Mathematics of Computation},
    VOLUME = {36},
      YEAR = {1981},
    NUMBER = {154},
     PAGES = {627--630},
      ISSN = {0025-5718,1088-6842},
   MRCLASS = {10A25 (10-04 65C05)},
  MRNUMBER = {606520},
       DOI = {10.2307/2007666},
       URL = {https://doi.org/10.2307/2007666},
}

@article{PollardRhoFactorization,
    AUTHOR = {Pollard, J. M.},
     TITLE = {A {M}onte {C}arlo method for factorization},
   JOURNAL = {Nordisk Tidskr. Informationsbehandling (BIT)},
  FJOURNAL = {Nordisk Tidskrift for Informationsbehandling},
    VOLUME = {15},
      YEAR = {1975},
    NUMBER = {3},
     PAGES = {331--334},
      ISSN = {0901-246X},
   MRCLASS = {10A25},
  MRNUMBER = {392798},
MRREVIEWER = {Stefan\ A.\ Burr},
       DOI = {10.1007/bf01933667},
       URL = {https://doi.org/10.1007/bf01933667},
}

@misc{mestrovic2014,
      title={Lucas' theorem: its generalizations, extensions and applications (1878--2014)}, 
      author={Romeo Meštrović},
      year={2014},
      eprint={1409.3820},
      archivePrefix={arXiv},
      primaryClass={math.NT},
      url={https://arxiv.org/abs/1409.3820}, 
}

@article{Odoni1985,
author = {Odoni, R. W. K.},
title = {The Galois Theory of Iterates and Composites of Polynomials},
journal = {Proceedings of the London Mathematical Society},
volume = {s3-51},
number = {3},
pages = {385-414},
doi = {https://doi.org/10.1112/plms/s3-51.3.385},
url = {https://londmathsoc.onlinelibrary.wiley.com/doi/abs/10.1112/plms/s3-51.3.385},
eprint = {https://londmathsoc.onlinelibrary.wiley.com/doi/pdf/10.1112/plms/s3-51.3.385},
year = {1985}
}

@article {BIJRMMST,
    AUTHOR = {Benedetto, Robert and Ingram, Patrick and Jones, Rafe and
              Manes, Michelle and Silverman, Joseph H. and Tucker, Thomas
              J.},
     TITLE = {Current trends and open problems in arithmetic dynamics},
   JOURNAL = {Bull. Amer. Math. Soc. (N.S.)},
  FJOURNAL = {American Mathematical Society. Bulletin. New Series},
    VOLUME = {56},
      YEAR = {2019},
    NUMBER = {4},
     PAGES = {611--685},
      ISSN = {0273-0979,1088-9485},
   MRCLASS = {37P05 (11G50 37P15 37P20 37P25 37P30 37P45 37P55)},
  MRNUMBER = {4007163},
       DOI = {10.1090/bull/1665},
       URL = {https://doi.org/10.1090/bull/1665},
}

@article {Juul2021,
    AUTHOR = {Juul, Jamie},
     TITLE = {The image size of iterated rational maps over finite fields},
   JOURNAL = {Int. Math. Res. Not. IMRN},
  FJOURNAL = {International Mathematics Research Notices. IMRN},
      YEAR = {2021},
    NUMBER = {5},
     PAGES = {3362--3388},
      ISSN = {1073-7928,1687-0247},
   MRCLASS = {37P05 (11R45 14E05)},
  MRNUMBER = {4227574},
MRREVIEWER = {Andrea\ Ferraguti},
       DOI = {10.1093/imrn/rnz217},
       URL = {https://doi.org/10.1093/imrn/rnz217},
}

@article {JKT2016,
    AUTHOR = {Juul, Jamie and Kurlberg, P\"ar and Madhu, Kalyani and Tucker,
              Tom J.},
     TITLE = {Wreath products and proportions of periodic points},
   JOURNAL = {Int. Math. Res. Not. IMRN},
  FJOURNAL = {International Mathematics Research Notices. IMRN},
      YEAR = {2016},
    NUMBER = {13},
     PAGES = {3944--3969},
      ISSN = {1073-7928,1687-0247},
   MRCLASS = {37P25 (11G35 37P15 37P35)},
  MRNUMBER = {3544625},
MRREVIEWER = {Joseph\ H.\ Silverman},
       DOI = {10.1093/imrn/rnv273},
       URL = {https://doi.org/10.1093/imrn/rnv273},
}

\end{document}